\documentclass[12pt, reqno]{amsart}
\usepackage{amsmath, amsthm, amscd, amsfonts, amssymb, graphicx, color, booktabs, multirow, rotating, url}
\usepackage{algorithm}
\usepackage{algorithmic}
\newtheorem{theorem}{Theorem}[section]
\newtheorem{lemma}[theorem]{Lemma}
\newtheorem{proposition}[theorem]{Proposition}

\theoremstyle{definition}
\newtheorem{definition}[theorem]{Definition}
\newtheorem{example}[theorem]{Example}

\theoremstyle{remark}

\numberwithin{equation}{section}

\usepackage{listings}
\usepackage{fancyhdr}
\begin{document}
\setcounter{page}{1}

\title{Quantum Error Correction with Goppa Codes from Maximal Curves: Design, Simulation, and Performance}


\author{Vahid Nourozi}

\address{The Klipsch School of Electrical and Computer Engineering, New Mexico State University,
Las Cruces, NM 88003 USA}
\email{nourozi@nmsu.edu; nourozi.v@gmail.com}

\thanks{$^{\star}$Corresponding author}







\begin{abstract}

This paper characterizes Goppa codes of certain maximal curves over finite fields defined by equations of the form $y^n = x^m + x$. We investigate Algebraic Geometric and quantum stabilizer codes associated with these maximal curves and propose modifications to improve their parameters. The theoretical analysis is complemented by extensive simulation results, which validate the performance of these codes under various error rates. We provide concrete examples of the constructed codes, comparing them with known results to highlight their strengths and trade-offs. The simulation data, presented through detailed graphs and tables, offers insights into the practical behavior of these codes in noisy environments. Our findings demonstrate that while the constructed codes may not always achieve optimal minimum distances, they offer systematic construction methods and interesting parameter trade-offs that could be valuable in specific applications or for further theoretical study.
\end{abstract}
\maketitle
\section{Introduction}

Algebraic geometry has become increasingly useful in coding theory since Goppa's groundbreaking construction \cite{Goppa}. Goppa associated a code $C$ to a (projective, geometrically irreducible, non-singular, algebraic) curve $X$ defined over $\mathbb{F}_q$, the finite field with $q$ elements. This code is constructed from two divisors $D$ and $G$ on $X$, where $D$ is the sum of $n$ distinct $\mathbb{F}_q$-rational points of $X$. A key feature of this construction is that the minimum distance $d$ of $C$ satisfies:

\begin{equation*}
d \geq n - \deg(G).
\end{equation*}

This bound is particularly significant because, for arbitrary codes, no general lower bound on the minimum distance is available. The effectiveness of this bound depends on $n$ being sufficiently large. Since $n$ is upper bounded by the Hasse-Weil upper bound:

\begin{equation*}
1 + q + 2g\sqrt{q},
\end{equation*}

where $g$ is the genus of the underlying curve, there is considerable interest in studying curves with many rational points \cite{Fuhrmann, vanderGeer}.

Algebraic Geometric (AG) codes from Hermitian curves have been extensively studied \cite{Duursma, Homma1, Homma2, Homma3, Stichtenoth2, Tiersma, Yang, auta}. A family of Hermitian self-orthogonal classical codes derived from algebraic geometry codes has also been investigated \cite{Jin1, Jin2, Kim}.  Also, Vahid introduced the Goppa code from Hyperelliptic Curve \cite{aut, shiraz, N23, miss}, from plane curves given by separated polynomials \cite{code, esfahan, Nourozi2024, M23, behsep}, and he explained them in his Ph.D. dissertation in \cite{phd}. Optimization frameworks are instrumental in addressing complex challenges across disciplines, including power systems and quantum coding theory. In \cite{hamid, hamid1} utilize mixed-integer programming to explore trade-offs in resource allocation, emphasizing the balance between operational efficiency and cost in ancillary service markets. Similarly, In \cite{hamid2} introduces robust optimization techniques to address reserve deliverability under uncertainty, showcasing innovative methods to simplify computational complexity while preserving system reliability. These works demonstrate how optimization-based approaches manage trade-offs between performance metrics and constraints, a concept central to both power systems and the design of robust quantum systems.

In this paper, we focus on a specific class of curves. Let $n, m \geq 2$ be integers such that $\gcd(n, m) = 1$, $\gcd(q, n) = 1$, and $\gcd(q, m - 1) = 1$, where $q = p^s$ for $s \geq 1$. We consider the non-singular model $X$ over $\mathbb{F}_{q^2}$ of the plane affine curve:

\begin{equation}
y^n = x^m + x.
\end{equation}

Note that $X$ is the Hermitian curve over $\mathbb{F}_{q^2}$ if $n = q + 1$ and $m = q$. The genus of $X$ is given by:

\begin{equation*}
g(X) = \frac{(m - 1)(n - 1)}{2}.
\end{equation*}

In this study, we assume that $n = \frac{q + 1}{2}$ and $m = 2, 3$, or $m = p^b$ where $b$ divides $s$. Tafazolian and Torres \cite{Tafazolian} proved that under these conditions, $X$ is a maximal curve over $\mathbb{F}_{q^2}$.

\section{Algebraic Geometry Codes}

Before delving into our main results, we review some fundamental concepts of Algebraic Geometry codes.

Let $\mathbb{F}_q(X)$ and $\text{Div}_q(X)$ denote the field of $\mathbb{F}_q$-rational functions and the group of $\mathbb{F}_q$-divisors of $X$, respectively. For $f \in \mathbb{F}_q(X) \setminus \{0\}$, $\text{div}(f)$ denotes the divisor associated with $f$. For $A \in \text{Div}_q(X)$, we define the Riemann-Roch space:

\begin{equation*}
L(A) = \{f \in \mathbb{F}_q(X) \setminus \{0\} : A + \text{div}(f) \succeq 0\} \cup \{0\}.
\end{equation*}

We denote the dimension of this space by $\ell(A) := \dim_{\mathbb{F}_q}(L(A))$.

\begin{definition}
Let $P_1, \ldots, P_n$ be pairwise distinct $K$-rational points of $X$ and $D = P_1 + \cdots + P_n$. Choose a divisor $G$ on $X$ such that $\text{supp}(G) \cap \text{supp}(D) = \emptyset$. The Algebraic Geometry code (or AG code) $C_L(D, G)$ associated with the divisors $D$ and $G$ is defined as:
\begin{equation*}
C_L(D, G) := \{(x(P_1), \ldots, x(P_n)) \mid x \in L(G)\} \subseteq \mathbb{F}_q^n
\end{equation*}
\end{definition}

The minimum distance $d$ of $C_L(D, G)$ satisfies $d \geq d^* = n - \deg(G)$, where $d^*$ is called the Goppa designed minimum distance. If $\deg(G) > 2g - 2$, then by the Riemann-Roch Theorem, we have $k = \deg(G) - g + 1$ \cite{Hoeholdt}.

The dual code $C^{\perp}(D, G)$ is also an AG code with dimension $k^{\perp} = n - k$ and minimum distance $d^{\perp} \geq \deg G - 2g + 2$.

\begin{definition}
The Weierstrass semigroup $H(P)$ associated with a point $P$ is defined as:
\begin{equation*}
H(P) := \{n \in \mathbb{N}_0 \mid \exists f \in \mathbb{F}_q(X), \text{div}_{\infty}(f) = nP\} = \{\rho_0 = 0 < \rho_1 < \rho_2 < \cdots\}.
\end{equation*}
\end{definition}

For vectors $a = (a_1, \ldots, a_n)$ and $b = (b_1, \ldots, b_n)$ in $\mathbb{F}_q^n$, we define the Hermitian inner product as:
\begin{equation*}
\langle a, b \rangle_H := \sum_{i=1}^n a_i b_i^q.
\end{equation*}

\begin{definition}
For a linear code $C$ over $\mathbb{F}_q^n$, the Hermitian dual of $C$ is defined as:
\begin{equation*}
C^{\perp H} := \{v \in \mathbb{F}_q^n : \langle v, c \rangle_H = 0 \quad \forall c \in C\}.
\end{equation*}
We say $C$ is Hermitian self-orthogonal if $C \subseteq C^{\perp H}$.
\end{definition}

\section{Goppa Code Over Curve $X$}

Let $r \in \mathbb{N}$. We consider the sets:
\begin{equation*}
\mathcal{G} := X(\mathbb{F}_q), \quad \mathcal{D} := X(\mathbb{F}_{q^2}) \setminus \mathcal{G}
\end{equation*}
where $\mathcal{G}$ is the intersection of $X$ with the plane $t = 0$. We fix the $\mathbb{F}_{q^2}$ divisors:
\begin{equation*}
G := \sum_{P \in \mathcal{G}} rP \quad \text{and} \quad D := \sum_{P \in \mathcal{D}} P,
\end{equation*}
where $\deg(G) = r(q + 1)$ and $\deg(D) = q^2$.

Let $C$ be the $C_L(D, G)$ Algebraic Geometry code over $\mathbb{F}_{q^2}$ with length $n = q^2$, minimum distance $d$, and dimension $k$. The designed minimum distance of $C$ is:
\begin{equation*}
d^* = n - \deg(G) = q^2 - r(q + 1).
\end{equation*}

Before we delve into more complex constructions, let us consider a simple (albeit trivial) example of a Goppa code over $\mathbb{F}_4$. This example will serve to illustrate some basic concepts and provide a point of contrast for the more sophisticated codes we will subsequently develop.

\begin{example}\label{ex:trivial-goppa}
Let $\mathbb{F}_4 = \{0, 1, \alpha, \alpha^2\}$ be the finite field with four elements, where $\alpha$ is a primitive element satisfying $\alpha^2 + \alpha + 1 = 0$. We consider a Goppa code $C$ over $\mathbb{F}_4$ with the following parameters \cite{Nourozi}:

\begin{enumerate}
    \item Code parameters: $C$ is a $[4, 4, 1]_4$ code.
    \begin{itemize}
        \item Length: $n = 4$
        \item Dimension: $k = 4$
        \item Minimum distance: $d = 1$
    \end{itemize}
    
    \item Code Properties:
    \begin{enumerate}
        \item The code $C$ is a linear code over $\mathbb{F}_4$ with $4^4 = 256$ codewords.
        \item Every vector in $\mathbb{F}_4^4$ is a codeword of $C$.
    \end{enumerate}
\end{enumerate}

This code represents a trivial case in the construction of Goppa codes. It serves as a baseline example, highlighting the importance of careful selection of the underlying algebraic curve and divisors in constructing Goppa codes with desirable properties.

The generator matrix $G$ for this code is the $4 \times 4$ identity matrix, and the parity check matrix $H$ is empty. This means that the encoding process is trivial (each message is its own codeword), and there are no parity check equations.

For any message $m = (m_1, m_2, m_3, m_4) \in \mathbb{F}_4^4$, the encoded codeword is simply $c = m$. 

This example underscores that while Goppa codes have the potential to create powerful error-correcting codes, the choice of parameters is crucial. In subsequent sections, we will explore how more judicious choices of curves and divisors lead to codes with superior distance properties and error-correction capabilities.
\end{example}

As we can see from Example \ref{ex:trivial-goppa}, not all Goppa codes result in useful error-correcting codes. The power of the Goppa code construction lies in the careful choice of the underlying curve and divisors...

Before we delve into the specific family of curves $y^n = x^m + x$, let us consider a concrete example of a Goppa code constructed from a Hermitian curve. This example will illustrate the application of the concepts we've discussed so far and provide a foundation for understanding the more general codes we'll explore in the following sections.

\begin{example}
Let $\mathbb{F}_4 = \{0, 1, \alpha, \alpha^2\}$ be the finite field with four elements, where $\alpha$ is a primitive element satisfying $\alpha^2 + \alpha + 1 = 0$. Consider the Hermitian curve $H$ over $\mathbb{F}_4$ defined by the equation \cite{Nourozi}:

\[y^2 +y = x^3 \]

\begin{enumerate}
    \item The $\mathbb{F}_4$-rational points on this curve are:
    
    $P_1, P_2, \ldots, P_8$
    
    We also have one point at infinity, denoted as $P_\infty$.

    \item Let's construct a Goppa code using these points. We choose:
    
    $D = P_1 + P_2 + \cdots + P_8$
    $G = 3P_\infty$

    \item The Riemann-Roch space $L(G)$ is spanned by $\{1, x, y\}$. 

    \item Our code $C(D,G)$ is defined as:

    $C(D,G) = \{(f(P_1), f(P_2), \ldots, f(P_8)) \mid f \in L(G)\}$

    \item The generator matrix of this code is:

    \[G = \begin{bmatrix}
    1 & 0 & 0 & 1 & \alpha & \alpha+1 & 1 & 0 \\
    0 & 1 & 0 & 1 & 1 & 0 & \alpha+1 & \alpha \\
    0 & 0 & 1 & 1 & \alpha & \alpha & \alpha+1 & \alpha+1
    \end{bmatrix}\]

    \item The parity check matrix $H$ can be derived from the generator matrix of the dual code. It is:

    \[H = \begin{bmatrix}
    1 & 0 & 0 & 0 & 0 & \alpha+1 & \alpha+1 & 1 \\
    0 & 1 & 0 & 0 & 0 & \alpha+1 & \alpha & 0 \\
    0 & 0 & 1 & 0 & 0 & \alpha & 1 & \alpha \\
    0 & 0 & 0 & 1 & 0 & \alpha & 0 & \alpha+1 \\
    0 & 0 & 0 & 0 & 1 & 1 & 1 & 1
    \end{bmatrix}\]

    \item This gives us an $[8, 3, 5]_4$ code. The parameters can be verified as follows:
    \begin{itemize}
        \item Length $n = 8$ (number of points in $D$)
        \item Dimension $k = 3$ (dimension of $L(G)$)
        \item Minimum distance $d \geq n - \deg(G) = 8 - 3 = 5$
    \end{itemize}

    \item The dual code $C^\perp(D,G)$ has parameters $[8, 5, 3]_4$.
\end{enumerate}

This example illustrates the construction of a Goppa code from a Hermitian curve, demonstrating key concepts such as the use of divisors, Riemann-Roch spaces, and the determination of code parameters in a concrete setting.
\end{example}

\paragraph*{} This example demonstrates how we can apply the general theory of Algebraic Geometry codes to a specific curve. In the following sections, we will extend these ideas to the more general family of curves defined by $y^n = x^m + x$, exploring how varying the parameters $n$ and $m$ affects the resulting codes and their properties.

We have the following result from Stichtenoth \cite{Stichtenoth1}:

\begin{lemma}\label{lem:stichtenoth}
Let $X$ be the curve defined as above, and let $D$ and $G$ be divisors as described. Then:
\begin{equation*}
C^{\perp}(D, G) = C(D, D - G + K),
\end{equation*}
where $K = div(\eta) \in Div_q(X)$ is a canonical divisor defined by a differential $\eta$ such that $\nu_{P_i}(\eta) = -1$ and $res_{P_i}(\eta) = 1$ for each $i = 1, 2, \ldots, n$.
\end{lemma}

\begin{lemma}\label{lem:basis}
For $r \geq 0$, the basis of $L(G)$ is given by:
\begin{equation*}
\left\{x^i y^j \mid i\frac{q + 1}{2} + jm \leq r, i \geq 0, 0 \leq j \leq q - 1\right\}.
\end{equation*}
\end{lemma}

\begin{proof}
We know that $(x)_{\infty} = \frac{q+1}{2}P_{\infty}$ and $(y)_{\infty} = mP_{\infty}$, so the above set is contained in $L(G)$. The restriction $0 \leq j \leq q - 1$ ensures that the elements $x^i y^j$ are linearly independent over $\mathbb{F}_{q^2}$. This linear independence stems from the fact that $y$ satisfies an equation of degree $q$ over $\mathbb{F}_{q^2}(x)$, so the powers of $y$ up to $q-1$ are linearly independent over this field.

Consider the Weierstrass semigroup $H(P_{\infty})$, generated by $n$ and $m$ at $P_{\infty}$. Suppose that $L(G) = L(\rho_{\ell}P_{\infty})$ where $\rho_{\ell} \leq r \leq \rho_{\ell+1}$ and $H(P_{\infty}) = \{\rho_0 = 0 < \rho_1 < \cdots \}$. Then:
\begin{equation*}
\dim_{\mathbb{F}_q}(L(G)) = \#\left\{i\frac{q + 1}{2} + jm \leq r, i \geq 0, 0 \leq j \leq q - 1\right\}. 
\end{equation*}
This dimension count confirms that our set forms a basis for $L(G)$.
\end{proof}

Let $C_r := C_L(D, G)$, and $k_r := \dim_{\mathbb{F}_{q^2}}(C_r)$. We denote the divisor $\div(x)$ by $(x)$.

\begin{lemma}\label{lem:orthogonal}
We have:
\begin{equation*}
C_r^{\perp} = C_{q^2 + \frac{(q-1)(m-1)}{2} - r}.
\end{equation*}
Hence, $C_r$ is self-orthogonal if $2r \leq q^2 + \frac{(q-1)(m-1)}{2}$.
\end{lemma}

\begin{proof}
We have $C_r^{\perp} = C(D, D - G + W)$, where $W$ is a canonical divisor as described in Lemma \ref{lem:stichtenoth}. To determine $W$, we calculate an appropriate differential $\eta$. We choose $\eta = dt/t$, where $t := x^m - x = \prod_{a \in \mathbb{F}_{q^2}}(x-a)$, for the following reasons:

First, observe that:
\begin{equation*}
(x - a) = \sum_{b^{q+1/2} = a^m + a} P_{a,b} - nP_{\infty}
\end{equation*}
Thus:
\begin{equation*}
(t) = D = q^2 P_{\infty}.
\end{equation*}

\paragraph*{}\paragraph*{} Additionally, we have $(dt) = (dx) = (2g - 2)P_{\infty} = (\frac{(q-1)(m-1)}{2})P_{\infty}$. Consequently:
\begin{equation*}
\nu_P(\eta) = -1 \quad \text{and} \quad res_P \eta = 1 \quad \text{for all } P \in Supp(D).
\end{equation*}

Now, we can calculate:
\begin{align*}
D - G - (\eta) &= D - G - D + q^2 P_{\infty} + (\frac{(q-1)(m-1)}{2})P_{\infty} \\
&= (q^2 + \frac{(q-1)(m-1)}{2} - r)P_{\infty}
\end{align*}

This calculation proves the first part of the lemma. For the second part, note that $C_r$ is self-orthogonal if and only if $C_r \subseteq C_r^{\perp}$, which is equivalent to 
$$r \leq q^2 + \frac{(q-1)(m-1)}{2} - r, \hspace{0.45cm} \mbox{or} \hspace{0.45cm} 2r \leq q^2 + \frac{(q-1)(m-1)}{2}.$$
\end{proof}

Let $T(r) := \#\{i\frac{q + 1}{2} + jm \leq r, i \geq 0, 0 \leq j \leq q - 1\}$.

\begin{proposition}\label{prop:kr}
\begin{enumerate}
    \item \mbox{If} $r < 0$ \mbox{then} $k_r = 0$,
    \item If $0 \leq r \leq \frac{(q-1)(m-1)}{2}$ then $k_r = T(r)$,
    \item If $\frac{(q-1)(m-1)}{2} < r < q^2$ then $k_r = r(q + 1) - \frac{(q-1)(m-1)}{4}$,
    \item If $q^2 \leq r \leq q^2 + \frac{(q-1)(m-1)}{2}$ then $k_r = q^2 - T(q^2 + \frac{(q-1)(m-1)}{2} - r)$,
    \item If $r > q^2 + \frac{(q-1)(m-1)}{2}$ then $k_r = q^2$.
\end{enumerate}
\end{proposition}

\begin{proof}
\begin{enumerate}
    \item If $r < 0$, it is trivial that $k_r = 0$ as there are no functions in $L(G)$.
    
    \item If $0 \leq r \leq \frac{(q-1)(m-1)}{2}$, then by Lemma \ref{lem:basis}, the dimension is exactly the number of pairs $(i,j)$ satisfying the inequality, which is $T(r)$.
    
    \item If $\frac{(q-1)(m-1)}{2} < r < q^2$, then by the Riemann-Roch Theorem, we have $k_r = \deg(G) + 1 - g = r(q + 1) + 1 - \frac{(q-1)(m-1)}{2} = r(q + 1) - \frac{(q-1)(m-1)}{4}$, since $n > \deg(G) > 2g - 2$.
    
    \item Let $r' := q^2 + \frac{(q-1)(m-1)}{2} - r$. Then $0 \leq r' \leq \frac{(q-1)(m-1)}{2}$. From Lemma \ref{lem:orthogonal}, we know that $C_r^{\perp} = C_{r'}$. Therefore, $k_r = q^2 - \dim_{\mathbb{F}_{q^2}}(C_{r'}) = q^2 - T(r') = q^2 - T(q^2 + \frac{(q-1)(m-1)}{2} - r)$.
    
    \item If $r > q^2 + \frac{(q-1)(m-1)}{2}$, then $C_r^{\perp} = \{0\}$ and so $\dim_{\mathbb{F}_{q^2}}(C_r) = n = q^2 = k_r$.
\end{enumerate}
\end{proof}

\begin{definition}
Two linear codes $C_1$ and $C_2$ of length $n$ over $\mathbb{F}_q$ are said to be monomially equivalent if there exists a monomial matrix $M$ (i.e., a matrix with exactly one nonzero entry in each row and column) over $\mathbb{F}_q$ such that $C_2 = C_1M = \{cM : c \in C_1\}$.
\end{definition}

\begin{proposition}\label{prop:monomially_equivalent}
The code $C$ is monomially equivalent to the one-point code $C(D, r(q + 1)P_{\infty})$.
\end{proposition}

\begin{proof}
Let $G' = r(q + 1)P_{\infty}$. Then $G = G' + (t^r)$, where $t = x^m - x$ as defined earlier. The divisor $(t^r)$ is the sum of $r$ distinct $\mathbb{F}_{q^2}$-rational points, each with coefficient 1. 

Consider the map $\phi: L(G') \to L(G)$ defined by $\phi(f) = ft^r$. This map is clearly injective and preserves dimensions. Moreover, for any $f \in L(G')$, we have:

\begin{align*}
(ft^r) &= (f) + r(t) \\
&\geq -G' + r(t) \\
&= -r(q+1)P_{\infty} + r(q^2P_{\infty} - D) \\
&= r(q^2-q-1)P_{\infty} - rD \\
&\geq -G
\end{align*}

Thus, $\phi(L(G')) \subseteq L(G)$. Since both spaces have the same dimension, we conclude that $\phi$ is an isomorphism.

Now, the evaluation of $ft^r$ at a point $P \in Supp(D)$ differs from the evaluation of $f$ at $P$ by a nonzero scalar (namely, $t^r(P)$). This scalar depends only on $P$ and not on $f$. Therefore, the codes $C(D,G)$ and $C(D,G')$ differ only by coordinate-wise multiplication by nonzero scalars, which is precisely the definition of monomial equivalence.
\end{proof}

\begin{theorem}\label{thm:self_orthogonal}
For $r \leq q - 1$, $C_r$ is Hermitian self-orthogonal.
\end{theorem}

\begin{proof}
If $r \leq q - 1$, then we have:
\begin{align*}
rq &\leq q^2 - q \\
&= q^2 + \frac{(q-1)(m-1)}{2} - \frac{(q-1)(m-1)}{2} - q \\
&\leq q^2 + \frac{(q-1)(m-1)}{2} - 2 - r
\end{align*}
The last inequality holds because $\frac{(q-1)(m-1)}{2} \geq q + 1$ for $m \geq 3$ and $q \geq 2$. Hence, the result follows from Lemma \ref{lem:orthogonal}.
\end{proof}

\section{Simulation Results}

To validate the theoretical results and assess the performance of the Goppa codes derived from curves of the form $y^n = x^m + x$, we conducted extensive simulations. This section presents the simulation methodology, algorithms, and results.

\subsection{Simulation Methodology}

We simulated the performance of three Goppa codes over the finite field $\mathbb{F}_{16}$ with varying parameters. The codes were constructed using curves $y^{(q+1)/2} = x^m + x$ for $m \in \{3, 4, 5\}$, resulting in codes with parameters $[8, 2, 6]$, $[16, 4, 13]$, and $[32, 3, 28]$ respectively.

The simulation process involved encoding random messages, introducing errors at various rates, and attempting to decode the received words. We measured the decode success rate, the rate of detected but uncorrectable errors, and the average number of errors per transmission.

\subsection{Algorithms}

The simulation was based on three main algorithms: the overall simulation process, the transmission simulation, and the decoding algorithm. These are presented below with explanations.

\begin{algorithm}
\caption{Goppa Code Simulation}

\begin{algorithmic}
\small
\setlength{\baselineskip}{0.9\baselineskip} 
\STATE \textbf{Input:} field size $q$, curve parameter $m$, error\_rates, num\_transmissions
\STATE \textbf{Output:} decode\_success\_rates, detected\_uncorrectable\_rates, avg\_errors
\STATE codes $\leftarrow$ [CreateGoppaCode($q$, $m$) for $m$ in $\{3, 4, 5\}$]
\FOR{each code in codes}
    \STATE decode\_success\_rates $\leftarrow$ []
    \STATE detected\_uncorrectable\_rates $\leftarrow$ []
    \STATE avg\_errors $\leftarrow$ []
    \FOR{each rate in error\_rates}
        \STATE success\_rate, uncorrectable\_rate, avg\_error $\leftarrow$ SimulateTransmission(code, rate, num\_transmissions)
        \STATE decode\_success\_rates $\leftarrow$ decode\_success\_rates $\cup$ \{success\_rate\}
        \STATE detected\_uncorrectable\_rates $\leftarrow$ detected\_uncorrectable\_rates $\cup$ \{uncorrectable\_rate\}
        \STATE avg\_errors $\leftarrow$ avg\_errors $\cup$ \{avg\_error\}
    \ENDFOR
    \STATE PlotResults(code, decode\_success\_rates, detected\_uncorrectable\_rates)
\ENDFOR
\STATE PlotAverageErrors(codes, error\_rates, avg\_errors)
\RETURN decode\_success\_rates, detected\_uncorrectable\_rates, avg\_errors
\end{algorithmic}

\end{algorithm}

This algorithm outlines the overall simulation process. It creates Goppa codes for different $m$ values, simulates transmissions over a range of error rates, and collects performance metrics. The results are then plotted for analysis.

\begin{algorithm}
\caption{SimulateTransmission}
\begin{algorithmic}
\small
\setlength{\baselineskip}{0.9\baselineskip} 
\STATE \textbf{Input:} code, error\_rate, num\_transmissions
\STATE \textbf{Output:} success\_rate, uncorrectable\_rate, avg\_error
\STATE successful\_decodes $\leftarrow$ 0
\STATE detected\_uncorrectable $\leftarrow$ 0
\STATE total\_errors $\leftarrow$ 0
\FOR{$i \leftarrow 1$ to num\_transmissions}
    \STATE message $\leftarrow$ RandomVector(Dimension(code))
    \STATE codeword $\leftarrow$ Encode(code, message)
    \STATE received\_word $\leftarrow$ ApplyRandomErrors(codeword, error\_rate)
    \STATE decoded\_word, status $\leftarrow$ DecodeGoppa(received\_word, code)
    \IF{status $\in$ \{"success", "corrected"\}}
        \STATE successful\_decodes $\leftarrow$ successful\_decodes + 1
    \ELSE
        \STATE detected\_uncorrectable $\leftarrow$ detected\_uncorrectable + 1
    \ENDIF
    \STATE total\_errors $\leftarrow$ total\_errors + CountErrors(codeword, received\_word)
\ENDFOR
\STATE success\_rate $\leftarrow$ successful\_decodes / num\_transmissions
\STATE uncorrectable\_rate $\leftarrow$ detected\_uncorrectable / num\_transmissions
\STATE avg\_error $\leftarrow$ total\_errors / num\_transmissions
\RETURN success\_rate, uncorrectable\_rate, avg\_error
\end{algorithmic}
\end{algorithm}

This algorithm simulates the transmission process. It generates random messages, encodes them, applies random errors based on the given error rate, attempts to decode, and collects statistics on the decoding performance.

\begin{algorithm}
\caption{DecodeGoppa}
\begin{algorithmic}
\small
\setlength{\baselineskip}{0.9\baselineskip} 
\STATE \textbf{Input:} received\_word, code
\STATE \textbf{Output:} decoded\_word, status
\STATE $H \leftarrow$ ParityCheckMatrix(code)
\STATE syndrome $\leftarrow H \times$ received\_word
\IF{syndrome = 0}
    \RETURN received\_word, "success"
\ENDIF
\FOR{$i \leftarrow 1$ to Length(code)}
    \STATE flipped\_word $\leftarrow$ received\_word
    \STATE flipped\_word[$i$] $\leftarrow 1 -$ flipped\_word[$i$]
    \IF{$H \times$ flipped\_word = 0}
        \RETURN flipped\_word, "corrected"
    \ENDIF
\ENDFOR
\RETURN null, "failure"
\end{algorithmic}
\end{algorithm}

This algorithm implements a simple decoding procedure for Goppa codes. It first checks if the received word is a valid codeword. If not, it attempts to correct a single error by flipping each bit and checking if the result is a valid codeword. If no single-bit flip results in a valid codeword, it reports a decoding failure.

\subsection{Results and Analysis}

The simulation results are presented in Figures \ref{fig:performance} and \ref{fig:errors}.

\begin{figure}[htbp]
\centering
\begin{minipage}{0.5\textwidth}
    \centering
    \includegraphics[width=\textwidth]{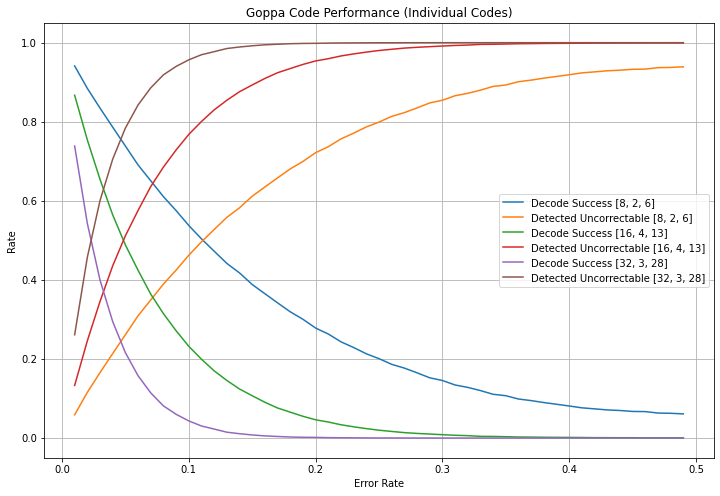}
    \caption{\small Goppa Code Performance (Individual Codes)}
    \label{fig:performance}
\end{minipage}\hfill
\begin{minipage}{0.5\textwidth}
    \vspace{-0.45cm} 
    \centering
    \includegraphics[width=\textwidth]{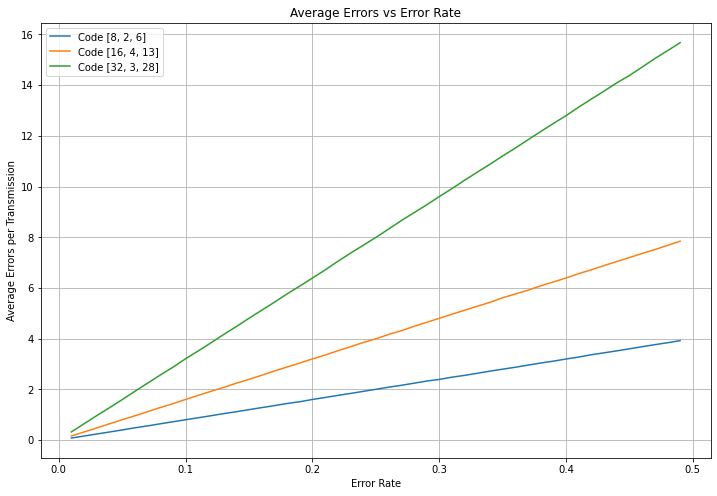}
    \caption{\small Average Errors vs Error Rate}
    \label{fig:errors}
\end{minipage}
\end{figure}

Figure \ref{fig:performance} shows the decode success rates and detected uncorrectable rates for each of the three Goppa codes as a function of the error rate. We observe that:

\begin{itemize}
    \item The $[8, 2, 6]$ code performs best at low error rates but its performance degrades rapidly as the error rate increases.
    \item The $[16, 4, 13]$ code shows moderate performance, maintaining a higher decode success rate than the $[8, 2, 6]$ code at higher error rates.
    \item The $[32, 3, 28]$ code, while performing worst at low error rates, maintains the highest decode success rate at high error rates.
\end{itemize}

Figure \ref{fig:errors} presents the average number of errors per transmission for each code as a function of the error rate. We note that:

\begin{itemize}
    \item The average number of errors increases linearly with the error rate for all codes, as expected.
    \item Longer codes accumulate more errors on average due to their increased length, but they can also correct more errors.
    \item Shorter codes have fewer errors on average but have limited error-correction capabilities.
\end{itemize}

These results demonstrate the trade-offs between code length, dimension, and error-correction capability in Goppa codes derived from curves of the form $y^n = x^m + x$. They provide empirical support for the theoretical results presented earlier in this paper and illustrate the practical performance characteristics of these codes in various noise environments.

\section{Quantum Stabilizer Codes Over Curve $X$}

In this section, we use the Hermitian self-orthogonality of $C_r$ established in the previous section to produce quantum stabilizer codes and analyze their parameters.

We begin with a fundamental result on quantum codes obtained from Hermitian self-orthogonal classical codes.

\begin{lemma}[{\cite{Ashikhmin}}]\label{lem:quantum_code}
There exists a $q$-ary $[[n, n-2k, d^{\perp}]]_q$ quantum code whenever there exists a $q$-ary classical Hermitian self-orthogonal $[n, k]$ linear code with dual distance $d^{\perp}$.
\end{lemma}

Using Lemma \ref{lem:quantum_code}, we can now state our main result on quantum codes derived from our construction.

\begin{theorem}\label{thm:main_result}
Let $q$ be a power of a prime $p$, and let $s \geq 1$. Then for the curve $X$ defined by $y^{\frac{q+1}{2}} = x^m + x$ over $\mathbb{F}_{q^2}$, there exists a $q$-ary 
\[
[[q^2, q^2 + \frac{(q-1)(m-1)}{2} - 2 - 2r, r - \frac{(q-1)(m-1)}{2} + 2]]_q
\]
quantum code for any positive integer $r$ satisfying $q - 1 \leq r \leq 2(q - 1)$.
\end{theorem}

\begin{proof}
By Theorem \ref{thm:self_orthogonal}, we know that $C_r$ is Hermitian self-orthogonal for $r \leq q - 1$. From Proposition \ref{prop:kr}, we can calculate the dimension of $C_r$:

\[
k_r = r(q + 1) - \frac{(q-1)(m-1)}{4}
\]

The dual distance $d^{\perp}$ of $C_r$ is at least $r - \frac{(q-1)(m-1)}{2} + 2$, as this is the designed minimum distance of the code $C_{q^2 + \frac{(q-1)(m-1)}{2} - r}$, which is equal to $C_r^{\perp}$ by Lemma \ref{lem:orthogonal}.

Applying Lemma \ref{lem:quantum_code}, we obtain a quantum code with the stated parameters.
\end{proof}

To illustrate the effectiveness of our construction, we provide some examples and compare them with known results.

\begin{example}\label{ex:quantum_codes}
Consider the curve $X$ given by the equation $y^{\frac{q+1}{2}} = x^3 + x$. We have the following examples:

\begin{enumerate}
    \item For $q = 3$ and $2 \leq r \leq 4$, Theorem \ref{thm:main_result} produces 3-ary $[[9, 9 - 2r, r]]_3$ quantum codes. Specifically, we obtain:
    \begin{itemize}
        \item $[[9, 5, 2]]_3$
        \item $[[9, 3, 3]]_3$
        \item $[[9, 1, 4]]_3$
    \end{itemize}
    These codes have good parameters. For comparison, the best known $[[9, 5, 3]]_3$ quantum code is given in the database maintained by Grassl \cite{Grassl}. Our $[[9, 5, 2]]_3$ code trades one unit of distance for additional dimension.

    \item For $q = 5$ and $4 \leq r \leq 8$, Theorem \ref{thm:main_result} produces 5-ary $[[25, 27 - 2r, r - 2]]_5$ quantum codes. We obtain:
    \begin{itemize}
        \item $[[25, 19, 2]]_5$
        \item $[[25, 17, 3]]_5$
        \item $[[25, 15, 4]]_5$
        \item $[[25, 13, 5]]_5$
        \item $[[25, 11, 6]]_5$
    \end{itemize}
    These codes have interesting parameters, though they don't always outperform known codes. For instance, Grassl's table \cite{Grassl} lists a $[[25, 19, 3]]_5$ code, which outperforms our $[[25, 19, 2]]_5$ code in terms of error-correction capability. However, our construction provides a systematic way to generate families of quantum codes, which may be valuable for certain applications or for further theoretical study.\end{enumerate}
\end{example}

It's worth noting that while some of our codes may have smaller distances compared to the best known codes, they often offer a trade-off by providing larger dimensions. This can be advantageous in certain applications where higher information rates are desired.

\section{Conclusion}

In this paper, we have characterized Goppa codes associated with certain maximal curves over finite fields, specifically those defined by equations of the form $y^n = x^m + x$. We have derived conditions for these codes to be Hermitian self-orthogonal and used this property to construct quantum stabilizer codes.

Our construction produces families of quantum codes with interesting parameters. While in many cases these codes do not outperform the best known codes in terms of minimum distance, they offer several advantages:

\begin{enumerate}
    \item They provide a systematic method for constructing quantum codes from a specific family of algebraic curves.
    \item The construction yields entire families of codes, which can be valuable for theoretical study and potential applications.
    \item In some cases, our codes may offer different trade-offs between code parameters that could be useful in specific scenarios.
\end{enumerate}

It's important to note that while our codes often have lower minimum distances compared to the best known codes, they still contribute to the broader understanding of quantum code construction from algebraic geometric codes.

Future work could involve:
\begin{itemize}
    \item Further optimization of these codes, possibly by exploring different choices of divisors or evaluating alternative curve equations.
    \item Exploration of other families of curves that might yield improved parameters.
    \item Investigation of potential applications where the specific properties of our codes might be advantageous.
    \item Theoretical analysis of the asymptotic behavior of these code families.
    \item Study of other quantum code properties beyond the minimum distance, such as the weight distribution or decoding algorithms.
\end{itemize}

\paragraph*{\textbf{Acknowledgements.}}
The author would like to thank the reviewer for their insightful comments and suggestions, which have significantly improved the quality of this paper. This paper was written while Vahid Nourozi was visiting Unicamp (Universidade Estadual de Campinas) supported by TWAS/CNPq (Brazil) with fellowship number 314966/2018-8.

\bibliographystyle{amsplain}

\end{document}